\documentclass[12pt]{article}
\usepackage{a4wide,amssymb,amsmath,amsthm,graphicx}
\usepackage[T1]{fontenc}
\usepackage{lmodern}
\usepackage[hidelinks]{hyperref}
\usepackage{url,needspace}
\newcommand{\doi}[1]{\href{https://doi.org/#1}{doi:\nolinkurl{#1}}}
\newtheorem{Thm}{Theorem}
\newtheorem{Lem}[Thm]{Lemma}
\newtheorem{Cor}[Thm]{Corollary}
\newtheorem{Prop}[Thm]{Proposition}
\newtheorem{Rem}[Thm]{Remark}
\newcommand{\beq}{\begin{equation}}
\newcommand{\eeq}{\end{equation}}
\begin{document}
\title{The Combinatorics of the Leading Root of the Partial Theta Function}
\author{Thomas Prellberg}
\date{22 September 2026}
\maketitle

\begin{abstract}
Let $x_0(q)=-\xi_0(q)$ be the leading formal root of
$\Theta_0(x,q)=\sum_{n\geq0}x^nq^{\binom n2}$. I give here explicit
combinatorial interpretations of the positive integer coefficients of
$\xi_0(q)=1+q+2q^2+4q^3+9q^4+\cdots$ in terms of rooted trees enriched by
stack polyominoes or certain Ferrers diagrams, weighted by total area.
The two enrichments may be chosen independently at each level of the tree.
A decomposition along the first-child path gives a combinatorial
interpretation of $1-\xi_0^{-1}$. By reserving two successor slots at the
root, I also obtain an interpretation of $1-\xi_0^{-2}$ and its zero
coefficient in degree three. The sequence decomposition gives a Lyndon-word
interpretation of the Euler-product exponents and proves their positivity
and weak monotonicity. Finally, I derive the coefficient asymptotic
$[q^n]\xi_0(q)\sim \xi_0(\rho)\rho^{-n}n^{-3/2}/(2\sqrt\pi)$,
where $\rho$ is the radius of convergence. The tree models are equinumerous
with the braid classes studied by Flores and Gonz\'alez-Meneses.
\end{abstract}

\section{Introduction}
\label{secintro}

Alan Sokal proved in \cite{Sokal2012} that the partial theta function \cite{Andrews2005,Andrews2009}
\beq
\label{theta}
\Theta_0(x,q)=\sum\limits_{n=0}^\infty x^nq^{\binom n2}
\eeq
admits a unique formal power series $\xi_0(q)\in \mathbb{Z}[[q]]$ satisfying $\Theta_0(-\xi_0(q),q)=0$, and that it has
strictly positive integer coefficients, 
\beq
\label{series}
\xi_0(q)=1+q+2q^2+4q^3+9q^4+21q^5+52q^6+133q^7+351q^8+948q^9+\ldots\;.
\eeq
I present here a combinatorial interpretation of these coefficients in terms of rooted trees with vertices enriched with a certain class of objects related to polyominoes; I will define these objects more carefully below. The enrichment is such that the out-degree of the vertices in the tree is matched by a particular geometric parameter of the polyominoes. I will extend these interpretations to an uncountable family of different representations.

The two tree interpretations and the mixed-level construction first appeared
in the 2012 version of this paper \cite{Prellberg2012}. The present revision
adds the inverse-power models, the Lyndon-word interpretation, and the
coefficient asymptotics.

Since that first version, Flores and Gonz{\'a}lez-Meneses
have found a different occurrence of the same coefficients in positive braid
monoids \cite{FloresGonzalezMeneses2022}. For $k\geq1$, $[q^k]\xi_0(q)$ counts
elements of length $k$ in $A_\infty$ whose maximal lexicographic representative
starts with $a_1$. More generally, $[q^k]\xi_0(q)^r$ allows the initial letter
to lie in $\{a_1,\ldots,a_r\}$; for $k=0$ the empty braid is included.
I give an ordered-forest interpretation of these powers in
Section~\ref{secforests}.

I also give combinatorial interpretations of Sokal's inverse-power series
and of the Euler-product exponents. For the inverse square, I use a family
of Ferrers diagrams with two reserved successor slots at the root
(Theorem~\ref{inverse-square}). Finally, I use the positive tree equation
to derive coefficient asymptotics (Theorem~\ref{asymptotics}).

Partial-theta identities and their relations with $q$-series are treated in \cite{Andrews2005,Andrews2009,Warnaar2003,WarnaarSurvey}.
The theory of their zeros and spectral values is developed in
\cite{KostovShapiro2013,Kostov2014,Kostov2016,KostovClosest2016}.
Kostov also studies the stabilization of the Laurent coefficients of the
higher zeros, with limiting coefficients enumerating three-coloured partitions \cite{KostovExpansions2015,KostovStabilization2015}. Recent work includes the spectral
exclusion regions of Gati and Kostov \cite{GatiKostov2026} and the complex
spectrum studied by Shapiro \cite{Shapiro2026}.
Some of this literature uses the normalization
$\theta(q,z)=\Theta_0(qz,q)$, which changes the stated locations of the zeros.

The partial theta function $\Theta_0(x,q)$ can be seen as a special case of the three-variable
Rogers-Ramanujan function
\beq
R(x,y,q)=\sum_{n=0}^\infty\frac{x^ny^{\binom n2}}{(1+q)(1+q+q^2)\ldots(1+q+\ldots+q^{n-1})}\;,
\eeq
whose root expansions motivate further positivity questions \cite[Sections 6 and 7]{Sokal2012}. Clearly $\Theta_0(x,q)=R(x,q,0)$, and note also that $R(x,q,q)$ has a particularly simple product representation,
\beq
R(x,q,q)=\prod_{n=0}^\infty(1+x(1-q)q^n)\;,
\eeq
whence one can write all roots explicitly as $-x_n(q)=q^{-n}+q^{-n+1}+q^{-n+2}+\ldots$
and the corresponding counting problem is immediate. Further positivity
questions about the roots of $R(x,y,q)$ are discussed in
\cite[Sections 6 and 7]{Sokal2012}. A general formal-root approach is
described by Fill and Sokal \cite{FillSokal2014}. I consider here the specialisation
$R(x,q,0)$.

\section{Functional equations and stack polyominoes}
\label{secstack}

The combinatorial interpretation of the leading root $x_0(q)$ of the partial theta function $\Theta_0(x,q)$ relies on the following characterisation of roots of the partial theta function (\ref{theta}), where the standard notation $(t;q)_n=\prod_{i=0}^{n-1}(1-tq^i)$ and $(t;q)_\infty=\prod_{i=0}^\infty(1-tq^i)$ is used.

\begin{Lem}[Eqns.~(3.2) and (4.9) in Sokal \cite{Sokal2012}]
\label{sokal}
The leading root $x_0(q)=-\xi_0(q)$ of the partial theta function $\Theta_0(x,q)$ satisfies
\beq
\label{root1}
\xi_0(q)=1+\sum_{n=1}^\infty\frac{q^n}{(q;q)_n(\xi_0(q)q;q)_{n-1}}
\eeq
and
\beq
\label{root2}
\xi_0(q)=1+\sum_{n=1}^\infty\frac{q^{n^2}\xi_0(q)^n}{(q;q)_n(\xi_0(q)q;q)_{n-1}}\;.
\eeq
\end{Lem}

\begin{proof}
Using a result of Sokal \cite[Lemma 2.1]{Sokal2012}, the partial theta function (\ref{theta}) satisfies
\begin{align}
\Theta_0(x,q)&=(q;q)_\infty(-x;q)_\infty\sum_{n=0}^\infty\frac{q^n}{(q;q)_n(-x;q)_n}\label{first}\\
&=(-x;q)_\infty\sum_{n=0}^\infty\frac{q^{n^2}(-x)^n}{(q;q)_n(-x;q)_n}\label{second}
\end{align}
as formal power series. From Eqn.~(\ref{first}) it follows that
\begin{multline}
\Theta_0(x,q)=(q;q)_\infty(-x;q)_\infty\sum_{n=0}^\infty\frac{q^n}{(q;q)_n(-x;q)_n}\\
=(q;q)_\infty(-xq;q)_\infty\left[1+x+\sum_{n=1}^\infty\frac{q^n}{(q;q)_n(-xq;q)_{n-1}}\right]\;.
\end{multline}
Hence $\Theta_0(-\xi_0(q),q)=0$ implies that
\beq
0=1-\xi_0(q)+\sum_{n=1}^\infty\frac{q^n}{(q;q)_n(\xi_0(q)q;q)_{n-1}}\;,
\eeq
and Eqn.~(\ref{root1}) follows. Similarly, from Eqn.~(\ref{second}) it follows that
\begin{multline}
\Theta_0(x,q)=(-x;q)_\infty\sum_{n=0}^\infty\frac{q^{n^2}(-x)^n}{(q;q)_n(-x;q)_n}\\
=(-xq;q)_\infty\left[1+x+\sum_{n=1}^\infty\frac{q^{n^2}(-x)^n}{(q;q)_n(-xq;q)_{n-1}}\right]\;.
\end{multline}
Hence $\Theta_0(-\xi_0(q),q)=0$ implies that
\beq
0=1-\xi_0(q)+\sum_{n=1}^\infty\frac{q^{n^2}\xi_0(q)^n}{(q;q)_n(\xi_0(q)q;q)_{n-1}}\;,
\eeq
and Eqn.~(\ref{root2}) follows.
\end{proof}

The functional equations (\ref{root1}) and (\ref{root2}) are, broadly speaking, similar
to those satisfied by different families of combinatorial trees. The interested reader
should examine Chapter VII.4 of \cite{Flajolet2009} and the more general theory of species \cite{BergeronLabelleLeroux1998} (due to Joyal \cite{Joyal1981}), for many examples of this.

\begin{figure}[t]
\begin{center}\includegraphics[width=0.8\textwidth]{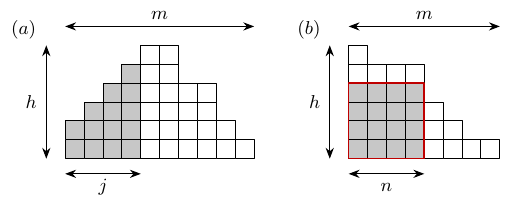}
\begin{minipage}{0.9\textwidth}\caption{(a) An example of a stack polyomino of width $m=10$,
rise $j=4$, height $h=6$ and area $A=37$. Columns associated to the rise are shaded in grey. 
(b) An example of a Ferrers diagram of width $m=8$, height $h=6$, and area $A=28$. The associated
Durfee square is shaded in grey, with a red border. 
\label{figstatferr}
}\end{minipage}\end{center}
\end{figure}

A combinatorial interpretation of Eqn.~(\ref{root1}) in Lemma~\ref{sokal} relates $\xi_0(q)$ to stack polyominoes. A {\em stack polyomino} is defined by a unimodal sequence of positive integers 
$h_1\leq h_2\leq h_3\leq\ldots\leq h_j<h_{j+1}\geq h_{j+2}\geq\ldots\geq h_m$. 
Given such a sequence, call $m$ the width, $j$ the rise ($j$ may be zero),
$h\equiv h_{j+1}$ the height, and $A=\sum_{i=1}^mh_i$ the area of the polyomino.
I draw a stack polyomino as a set of columns of height $h_1,\ldots,h_m$ from left to right;
an example is given in Figure \ref{figstatferr}(a). Viewed in such a way, a unimodal sequence is 
equivalent to a stack polyomino of horizontal rows with no overhangs, hence the name.

Similarly, a combinatorial interpretation of Eqn.~(\ref{root2}) in Lemma~\ref{sokal} will relate $\xi_0(q)$
to Ferrers diagrams. A \emph{Ferrers diagram} is defined by a decreasing sequence of positive integers $m_1\geq m_2\geq m_3\geq\ldots\geq m_h$.
Given such a sequence, call $h$ the height, $m=m_1$ the width, and $A=\sum_{i=1}^hm_i$ the area of the Ferrers diagram. I draw Ferrers diagrams as a set of rows of lengths $m_1,\ldots, m_h$ from bottom to top; an example is shown in Figure \ref{figstatferr}(b). (The example shown in Figure \ref{figstatferr}(b) also has $m_n=n$ for some integer $n$, a property that will become relevant later.)

The next lemma is an extension of results given in \cite{PrellbergOwczarek1995}. For general techniques for enumerating column-convex polyominoes, see also \cite{BousquetMelou1996}. 

\begin{Lem}
\label{lemstackgen}
The generating function $G(x,y,a,q)$ of stack polyominoes enumerated with respect to width ($x$), height ($y$), rise ($a$), and total area ($q$), is given by
\beq
G(x,y,a,q)=\sum_{n=1}^\infty\frac{x(yq)^n}{(xq;q)_n(axq;q)_{n-1}}\;.
\eeq
\end{Lem}

\begin{proof}
Any stack polyomino of height $h>1$ is uniquely constructed by appending a row to the bottom of a stack polyomino of height $h-1$, as indicated in Figure \ref{figstack}. 

\begin{figure}[t]
\begin{center}\includegraphics[width=0.8\textwidth]{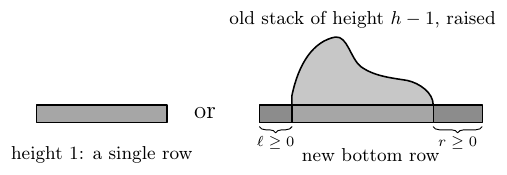}
\begin{minipage}{0.9\textwidth}\caption{Construction of stack polyominoes: a stack polyomino either has height one, hence consists of a single row, or, if its height is greater than one, can be obtained by placing a stack polyomino of height one less on top of a new bottom row. The new bottom row may extend an arbitrary non-negative distance to the left and to the right.
\label{figstack}
}\end{minipage}\end{center}
\end{figure}

This construction leads to a functional equation for the generating function $G(x,y,a,q)$. Stack polyominoes of height one are
counted by 
\beq
xyq+x^2yq^2+x^3yq^3+\ldots=\frac{xyq}{1-xq}\;.
\eeq
Appending a row to the bottom of a stack polyomino of height $h-1$ is equivalent
to increasing each column height of that polyomino by one and appending rows of non-negative length to each side of the new bottom row. Increasing each column 
height corresponds to a change of $G(x,y,a,q)$ to 
\beq
yG(xq,y,a,q)\;,
\eeq 
Appending rows to the left and right corresponds to multiplication with
\beq
1+axq+a^2x^2q^2+\ldots=\frac1{1-axq}
\eeq
and
\beq
1+xq+x^2q^2+\ldots=\frac1{1-xq}\;,
\eeq
respectively. Taken together, one finds that the generating
function $G(x)=G(x,y,a,q)$ satisfies the functional equation
\beq
G(x)=\frac{xyq}{1-xq}+\frac y{(1-xq)(1-axq)}G(xq)\;.
\eeq
Iteration of this functional equation leads to the desired result.
\end{proof}

Combining the functional equation (\ref{root1}) of Lemma~\ref{sokal} and the combinatorial statement of Lemma \ref{lemstackgen} provides a connection of the root of the partial theta function with the enumeration of stack polyominoes.

\begin{Cor}
\label{corstack}
The leading root $x_0(q)=-\xi_0(q)$ of the partial theta function $\Theta_0(x,q)$ satisfies
\beq
\xi_0(q)=F(\xi_0(q),q)\;,
\eeq
where $F(a,q)=1+G(1,1,a,q)$ is the generating function of stack polyominoes augmented by the `empty polyomino' of weight one, enumerated with respect to the rise ($a$) and total area ($q$).
\end{Cor}

\begin{proof}
Using Lemma~\ref{sokal} and Lemma \ref{lemstackgen} one finds immediately
\beq
\xi_0(q)=1+\sum_{n=1}^\infty\frac{q^n}{(q;q)_n(\xi_0(q)q;q)_{n-1}}=1+G(1,1,\xi_0(q),q)\;.
\eeq
\end{proof}

\section{The two enriched-tree models}
\label{sectrees}

For an interpretation of this result, I now turn to the theory of species \cite{BergeronLabelleLeroux1998},
and in particular to the combinatorics of Lagrange inversion in the context of combinatorial functional equations
\cite[Chapter 3]{BergeronLabelleLeroux1998}. The generating functions aspect of this theory is
very well explained in \cite[Section VII.4]{Flajolet2009}. 

Given a combinatorial species of structures $R$, an \emph{$R$-enriched rooted tree} on a finite set $U$ is the data of
(i) an arbitrary rooted tree on $U$, and (ii) an $R$-structure on the fiber of each vertex $u\in U$, in
this rooted tree. Here, the \emph{fiber} of a vertex $u$ refers to the (possibly empty) set of immediate successors of $u$, when all edges of the rooted tree are oriented away from the root. The out-degree of $u$ is the cardinality of its fiber.

For the two species used here, an enriching structure assigns the successor
labels to slots ordered from left to right. The corresponding unlabelled
objects are therefore \emph{ordered rooted trees}: the children of each
vertex are linearly ordered, and its decoration supplies exactly that many
slots. I shall count these trees without labels, using decomposition at the
root.

The shape species below are graded by area, with finitely many decorations
of each fixed area and slot cardinality. Their generating functions are
interpreted coefficientwise in $q$. To apply an ordinary finite-species
construction, restrict the decorations to area at most $M$ and then let
$M$ increase. For trees of total area $A$, every decoration has area at
most $A$, so the coefficient of $q^A$ stabilizes for $M\geq A$.
The bound on the number of vertices given below also ensures that this
coefficient is finite.

I shall make use of the following result.

\begin{Lem}[{\cite[Section~3.1]{BergeronLabelleLeroux1998}}]
Let $R$ be a species of structures. Then the species ${\cal A}_R$ of $R$-enriched rooted trees is uniquely determined,
up to isomorphism, by the combinatorial equation
\beq
{\cal A}_R=X\cdot R({\cal A}_R)\;.
\eeq
\end{Lem}

\begin{proof}
Deleting the root leaves a set of rooted subtrees, with an $R$-structure
on their roots. Conversely, attaching these roots to a new root recovers
an $R$-enriched tree. These operations are inverse. Uniqueness follows
by induction on the number of vertices, since each subtree has fewer
vertices than the original tree.
\end{proof}

\begin{figure}[t]
\begin{center}\includegraphics[width=0.4\textwidth]{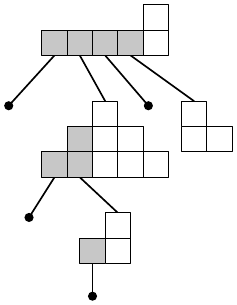}
\begin{minipage}{0.9\textwidth}\caption{An example of an $S_q$-enriched rooted tree with 8 vertices and total area 21. Note that each column of the rise of the stack polyomino (shown in grey) is associated with a successor vertex in the tree. A leaf of the tree can be decorated by a stack polyomino of rise zero (a Ferrers diagram), or by an empty stack (shown here as a black dot).
\label{fig2}
}\end{minipage}\end{center}
\end{figure}

I now consider the species $S_q$ of stack polyominoes augmented by the `empty polyomino', weighted according to their area by the generating variable $q$, and with size given by the rise of the stack polyomino. $S_q$-enriched rooted trees are ordered rooted trees such that each vertex of the tree with out-degree $d$ is decorated by a stack polyomino with rise $d$, and the weight of the tree is given by $q^A$ where $A$ is the sum of the areas of all of these polyominoes. Figure \ref{fig2} shows an example of an $S_q$-enriched rooted tree.

\begin{figure}[t]
\begin{center}\includegraphics[width=0.5\textwidth]{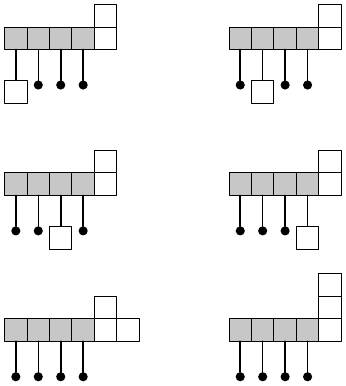}
\begin{minipage}{0.9\textwidth}\caption{All six $S_q$-enriched rooted trees with 5 vertices and total area 7. 
\label{figexample1}
}\end{minipage}\end{center}
\end{figure}

The area of a stack decoration is at least its rise; likewise the area of
a Ferrers decoration is at least its width. Consequently, in either model,
or in a mixture of the models, a tree with $v$ vertices and total area $A$
satisfies $v-1\leq A$. There are only finitely many such trees of each area:
there are finitely many plane trees with at most $A+1$ vertices and finitely
many decorations of bounded area. Thus the vertex variable may be set to
one coefficient by coefficient.

I shall also use the following formal uniqueness argument. Here
$\widetilde F(a,q)$ denotes the right hand side of \eqref{root2} with
$\xi_0(q)$ replaced by $a$. Write $\nu_q$ for $q$-adic order, with
$\nu_q(0)=\infty$. In the stack map $F(a,q)$ the least
area of a decoration with positive rise is three. In the Ferrers map
$\widetilde F(a,q)$ the least area of a positive-width decoration is one.
The two sums in Lemma~\ref{sokal} therefore imply, for $U,V\in\mathbb Z[[q]]$,
\begin{equation}\label{contractions}
 \begin{split}
 \nu_q\bigl(F(U,q)-F(V,q)\bigr)&\geq\nu_q(U-V)+3,\\
 \nu_q\bigl(\widetilde F(U,q)-\widetilde F(V,q)\bigr)&\geq\nu_q(U-V)+1.
 \end{split}
\end{equation}
The positive-width interpretation of the second inequality follows from
Lemma~\ref{lemferrers}.
Both maps have polynomial coefficients in their first variable at each
$q$-degree. Factoring $U^d-V^d$ in each positive-degree term proves
\eqref{contractions}. Each map thus has at most one formal fixed point,
and Lemma~\ref{sokal} identifies that fixed point as $\xi_0$.

I can now state the first tree interpretation.

\begin{Thm}
\label{thmA}
Let $S_q$ be the species of stack polyominoes augmented by the `empty polyomino', weighted by area ($q$), with size given by the rise.
Then $\xi_0(q)$ enumerates $S_q$-enriched rooted trees, weighted 
with respect to the total area of the stack polyominoes at the vertices of the tree.
\end{Thm}

\begin{proof}
The species ${\cal A}_{S_q}$ of $S_q$-enriched rooted trees satisfies 
${\cal A}_{S_q}=X\cdot S_q({\cal A}_{S_q})$.
Correspondingly, its generating function $A(t,q)$ satisfies
\beq
A(t,q)=tF(A(t,q),q)\;,
\eeq
where $t$ is the generating variable for the number of vertices in the tree. For $t=1$, this equation reduces to $A(1,q)=F(A(1,q),q)$, whence one identifies $\xi_0(q)=A(1,q)$ by \eqref{contractions} and Lemma~\ref{sokal}.
\end{proof}

The generating function $A(t,q)$ is a refinement of $\xi_0(q)$, enumerating enriched trees with respect to area and number of vertices, which could be studied in its own right. One computes easily
\begin{multline}
A(t,q)=t+tq+2tq^2+(t^2+3t)q^3+(t^3+3t^2+5t)q^4+(t^4+4t^3+9t^2+7t)q^5\\
+(t^5+5t^4+15t^3+20t^2+11t)q^6+(t^6+6t^5+23t^4+44t^3+44t^2+15t)q^7+\ldots\;.
\end{multline}
For example, the occurrence of the monomial $6t^5q^7$ in $A(t,q)$ indicates there are six $S_q$-enriched rooted trees with 5 vertices and total area 7. These are shown in Figure \ref{figexample1}.

As an immediate consequence of Theorem \ref{thmA}, one obtains a monotonicity result for the coefficients of $\xi_0(q)$.

\begin{Cor}
The coefficients of the power series $\xi_0(q)$ are non-decreasing.
\end{Cor}

\begin{proof}
Monotonicity follows from the existence of an injection of $S_q$-enriched rooted trees with total area $A$ to $S_q$-enriched rooted trees with total area $A+1$, which is defined as follows. 

The only $S_q$-enriched rooted tree with zero total area is a tree with a single vertex, enriched
by the empty stack polyomino. Map this tree to an $S_q$-enriched rooted tree with total area one by enriching its vertex 
by a square. The one-square decoration has rise zero. 

Now consider an $S_q$-enriched rooted tree with non-zero total area $A$. Then the root of this tree is 
enriched by a stack polyomino of non-zero area. Map this tree to an $S_q$-enriched rooted tree with total area $A+1$
by appending a square to the right of the bottom row of the stack polyomino at its root. This increases its area by 
one, and as this operation does not affect the rise of this stack polyomino, the resulting tree 
is again an $S_q$-enriched rooted tree.

Clearly, different $S_q$-enriched trees with total area $A$ get mapped to different $S_q$-enriched trees with total area $A+1$, hence this mapping is injective.
\end{proof}

\subsection{Admissible Ferrers diagrams}

A combinatorial interpretation of Eqn.~(\ref{root2}) in Lemma~\ref{sokal} provides another description of $\xi_0(q)$, relating it to partitions of integers with the
particular property that there exists an integer $n$ such that the $n$-th part 
of the partition has exactly size $n$. 

Recall that the generating function for Ferrers diagrams enumerated with respect to width ($x$), height ($y$), and area ($q$) is given by
\beq
\label{twoferrersums}
H(x,y,q)=\sum_{n=1}^\infty\frac{x^nyq^n}{(yq;q)_n}=\sum_{n=1}^\infty\frac{(xy)^nq^{n^2}}{(xq;q)_n(yq;q)_n}\;.
\eeq
For fixed width $n$, the compulsory bottom row contributes $x^nyq^n$; the remaining rows have lengths at most $n$, giving the first expression. The second expression is obtained by constructing Ferrers diagrams by adding Ferrers diagrams to the top and right of a square of size $n$, which is also known as the Durfee square \cite{Stanley2012} associated to that Ferrers diagram. A close inspection of Eqn.~(\ref{root2}) indicates that a modification of the second sum in Eqn.~(\ref{twoferrersums}) is needed, using a construction where the single square immediately to the right of the top row of the Durfee square is absent from the Ferrers diagram. This condition is equivalent to saying that the $n$-th largest row of the Ferrers diagram has length $n$ for some positive integer $n$.

\begin{Lem}
\label{lemferrers}
The generating function $\tilde G(x,y,q)$ of Ferrers diagrams with $n$-th largest row having length $n$ for some positive integer $n$, enumerated with respect to width ($x$), height ($y$), and total area ($q$), is given by 
\beq
\tilde G(x,y,q)=\sum_{n=1}^\infty\frac{(xy)^nq^{n^2}}{(yq;q)_n(xq;q)_{n-1}}\;.
\eeq
\end{Lem}

\begin{proof}
Partitioning Ferrers diagrams by their associated Durfee square, the side length of which is given by
\beq
n=\max\{i:m_i\geq i\}\;,
\eeq
one obtains Ferrers diagrams for which the $n$-th largest row has precisely length $n$ by adding Ferrers diagrams of width $\leq n$ to the top of the square
and adding Ferrers diagrams of height $\leq n-1$ to the right of the square. 

The weight of a Durfee square of fixed side length $n$ is $(xy)^nq^{n^2}$. Adding Ferrers diagrams of height $\leq n-1$ to the right hand side of the square corresponds to
multiplication with
\beq
\frac1{1-xq^{n-1}}\cdot\frac1{1-xq^{n-2}}\cdot\ldots\cdot\frac1{1-xq^2}\cdot\frac1{1-xq}=\frac1{(xq;q)_{n-1}}\;.
\eeq
Here, a factor $1/(1-xq^k)$ corresponds to the addition of an arbitrary number of columns of height $k$, and clearly $k$ can range from $1$ to $n-1$. Similarly, adding
Ferrers diagrams of width $\leq n$ to the top of the square corresponds to multiplication
with
\beq
\frac1{1-yq^{n}}\cdot\frac1{1-yq^{n-1}}\cdot\ldots\cdot\frac1{1-yq^2}\cdot\frac1{1-yq}=\frac1{(yq;q)_{n}}\;.
\eeq
Here, a factor $1/(1-yq^k)$ corresponds to the addition of an arbitrary number of rows of width $k$, and clearly $k$ can range from $1$ to $n$. A Durfee square of fixed side length $n$ gives therefore rise to an associated generating function 
\beq
\tilde G_n(x,y,q)=\frac{(xy)^nq^{n^2}}{(yq;q)_n(xq;q)_{n-1}}\;,
\eeq
and summing over all side lengths $n\geq1$ gives
\beq
\tilde G(x,y,q)=\sum_{n=1}^\infty\tilde G_n(x,y,q)
\eeq
as desired.
\end{proof}

\begin{figure}[t]
\begin{center}\includegraphics[width=0.82\textwidth]{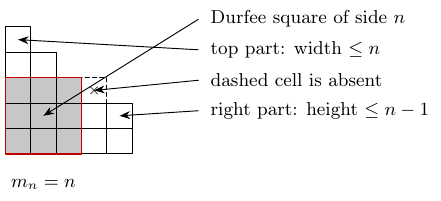}
\begin{minipage}{0.9\textwidth}\caption{The Durfee-square decomposition used in Lemma \ref{lemferrers}. The condition $m_n=n$ is equivalent to the absence of the square immediately to the right of the top row of the Durfee square.
\label{figferrersdecomposition}
}\end{minipage}\end{center}
\end{figure}

Combining the functional equation (\ref{root2}) of Lemma~\ref{sokal} and the combinatorial statement of Lemma \ref{lemferrers} provides a connection of the root of the partial theta function with the enumeration of Ferrers diagrams.

To achieve this, I consider the species $F_q$ of Ferrers diagrams with $n$-th largest row having length $n$ for some integer $n$, 
augmented by the `empty polyomino', weighted according to their area by the generating variable $q$, and with size given by the 
width of the polyomino. $F_q$-enriched rooted trees are ordered rooted trees such that each vertex of the tree with out-degree $d$ 
is decorated by a Ferrers diagram of width $d$, and the weight of the tree is given by $q^A$ where $A$ is the sum of the areas of 
all of these polyominoes. Figure \ref{fig2a} shows an example of an $F_q$-enriched rooted tree.

\begin{Cor}
The leading root $x_0(q)=-\xi_0(q)$ of the partial theta function $\Theta_0(x,q)$ satisfies
\beq
\xi_0(q)=\tilde F(\xi_0(q),q)\;,
\eeq
where $\tilde F(a,q)=1+\tilde G(a,1,q)$ is the generating function of Ferrers diagrams with $n$-th largest row having length $n$ for some integer $n$, augmented by the `empty polyomino' and enumerated with respect to width ($a$) and total area ($q$).
\end{Cor}

\begin{proof}
Combining Lemma~\ref{sokal} and Lemma \ref{lemferrers} one finds
\beq
\xi_0(q)=1+\sum_{n=1}^\infty\frac{\xi_0(q)^nq^{n^2}}{(q;q)_n(\xi_0(q)q;q)_{n-1}}=1+\tilde G(\xi_0(q),1,q)\;.
\eeq
\end{proof}

\begin{figure}[t]
\begin{center}\includegraphics[width=0.36\textwidth]{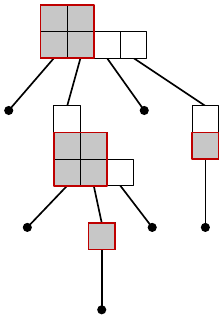}
\begin{minipage}{0.9\textwidth}\caption{An example of an $F_q$-enriched rooted tree with 10 vertices and total area 15. Note that the out-degree of each vertex is equal to the width of the Ferrers diagram associated with that vertex. The occurring Ferrers diagrams are such that their $n$-th largest row has length $n$ for some integer $n$ (the associated Durfee square is indicated in grey, with a red border). In particular, all leaves are decorated by empty Ferrers diagrams.
\label{fig2a}
}\end{minipage}\end{center}
\end{figure}

This leads to the second theorem of this paper, which gives a different explicit combinatorial interpretation of the coefficients of $\xi_0(q)$. (Of course this theorem also implies that these coefficients are positive integers.)

\begin{Thm}
\label{thmB}
Let $F_q$ be the species of Ferrers diagrams with $n$-th largest row having length $n$ for some integer $n$, weighted by area ($q$), with size 
given by the width of the Ferrers diagram, augmented by the `empty polyomino'.
Then $\xi_0(q)$ enumerates $F_q$-enriched rooted trees with respect to the total area of the Ferrers diagrams at the vertices of the tree.
\end{Thm}

\begin{proof}
The species ${\cal A}_{F_q}$ of $F_q$-enriched rooted trees satisfies ${\cal A}_{F_q}=X\cdot F_q({\cal A}_{F_q})$.
Correspondingly, its generating function $\tilde A(t,q)$ satisfies
\beq
\tilde A(t,q)=t\tilde F(\tilde A(t,q),q)\;,
\eeq
where $t$ is the generating variable for the number of vertices in the trees. For $t=1$, this equation reduces to 
$\tilde A(1,q)=\tilde F(\tilde A(1,q),q)$, and formal uniqueness in \eqref{contractions} gives $\xi_0(q)=\tilde A(1,q)$.
\end{proof}

Figure \ref{fig2a} shows an example of an $F_q$-enriched rooted tree. The generating function $\tilde A(t,q)$ is a refinement of $\xi_0(q)$ clearly different from $A(t,q)$. One computes easily
\begin{multline}
\tilde A(t,q)=t+t^2q+(t^3+t^2)q^2+(t^4+2t^3+t^2)q^3+(t^5+3t^4+4t^3+t^2)q^4+(t^6+4t^5+10t^4+5t^3+t^2)q^5\\
+(t^7+5t^6+21t^5+17t^4+7t^3+t^2)q^6+(t^8+6t^7+41t^6+47t^5+29t^4+8t^3+t^2)q^7+\ldots
\end{multline}
For example, the occurrence of the monomial $8t^3q^7$ in $\tilde A(t,q)$ indicates there are eight $F_q$-enriched rooted trees with 3 vertices and total area 7. These are shown in Figure \ref{figexample2}.

\begin{figure}[t]
\begin{center}\includegraphics[width=0.68\textwidth]{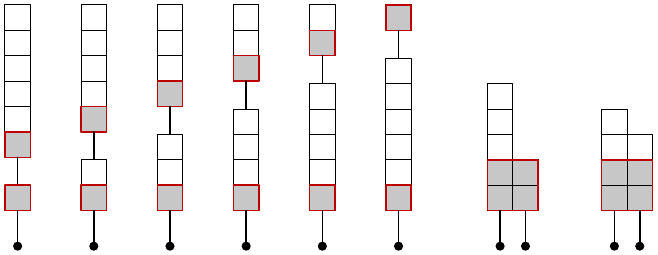}
\begin{minipage}{0.9\textwidth}\caption{All eight $F_q$-enriched rooted trees with 3 vertices and total area 7. 
\label{figexample2}
}\end{minipage}\end{center}
\end{figure}

\section{Inverse powers and a decomposition into sequences}
\label{secinverse}

Write $T(q)=\xi_0(q)$, and let $g_w(q)$ be the area generating function of the
non-empty admissible Ferrers diagrams of width $w$. Here and below, ``admissible''
means that the $n$-th largest row has length $n$ for some positive integer $n$.
Then
\begin{equation}\label{widthsplit}
 \widetilde F(z,q)=1+\sum_{w\geq1}g_w(q)z^w,
 \qquad T=1+\sum_{w\geq1}g_w T^w.
\end{equation}
A width-one diagram has positive area, and an admissible diagram of
width $w\geq2$ has area at least $w+2$, attained by the two-row diagram
$(w,2)$. Thus each coefficient of $q$ is a polynomial in $z$. The
substitutions and divisions by $T$, whose constant term is one, are
therefore valid formally.

Sokal proved that every positive-degree coefficient of $1-T^{-1}$ is
strictly positive, and that the same holds for $1-T^{-2}$ except for its
zero coefficient of $q^3$ \cite[Theorems 1.2 and 1.3]{Sokal2012}.
He asked for combinatorial interpretations in Remark~2 of each of
Sections~4 and~5. I now give such interpretations, recovering these
positivity results. Throughout this section, let $\mathcal T$
denote the class of $F_q$-enriched trees of Theorem~\ref{thmB}, referred to
below as ordinary trees. Its unique area-zero object is a single vertex
with empty decoration.

\subsection{One reserved slot}

Let $\mathcal B$ consist of a non-empty admissible Ferrers root decoration of
width $w$, with its \emph{first} successor slot left vacant and ordinary
subtrees attached to its remaining $w-1$ slots. I call this vacant first
slot the reserved slot.

\Needspace{11\baselineskip}
\begin{Prop}\label{propinverseone}
The class $\mathcal B$, weighted by total area, has generating function
\begin{equation}\label{inverseone}
 B(q)=\sum_{w\geq1}g_w(q)T(q)^{w-1}=1-T(q)^{-1}.
\end{equation}
There is a weight-preserving bijection
\begin{equation}\label{spinesequence}
 \mathcal T\simeq\operatorname{SEQ}(\mathcal B).
\end{equation}
In particular, this recovers Sokal's strict positivity result for the
positive-degree coefficients of $B$.
\end{Prop}

\begin{proof}
The generating-function identity follows by subtracting one from
\eqref{widthsplit} and dividing by $T$. For a direct decomposition, start at the
root of a tree in $\mathcal T$ and follow the first child repeatedly. Every
non-empty decoration has positive width and hence at least one child. Thus the
path ends at an empty-decorated leaf. At each preceding vertex cut the edge to
its first child, retaining the root decoration and the other child subtrees.
Read from the root towards the terminal leaf, the resulting pieces form a
sequence of $\mathcal B$-objects.

Conversely, join an ordered sequence of these pieces along their reserved
first slots, and put an empty-decorated leaf in the last reserved slot. The
empty sequence gives the single empty-decorated tree. These operations are
inverse and preserve area, proving \eqref{spinesequence} and hence
$T=1/(1-B)$. Every piece has positive area, so the sequence construction is
coefficientwise finite. Finally, a column of height $m$, with its only slot
reserved, is a $\mathcal B$-object of area $m$ for every $m\geq1$.
\end{proof}

The first coefficients are
\begin{equation}
 B(q)=q+q^2+q^3+2q^4+4q^5+10q^6+25q^7+66q^8+\cdots.
\end{equation}
The identity $T=1+BT$ also gives the monotonicity of the coefficients
of $T$ directly from Sokal's Theorem~1.2: writing $L_n=[q^n]T$ and
$b_m=[q^m]B$, for $n\geq1$ one has
$L_n=\sum_{m=1}^n b_mL_{n-m}\geq L_{n-1}$, since $b_1=1$.

\subsection{Two reserved slots and the inverse square}

Let $\mathcal D$ consist of an admissible Ferrers root of width $w\geq2$,
with its first two successor slots reserved and ordinary subtrees in its other
$w-2$ slots. Let $\mathcal U$ be the subclass of $\mathcal B$ whose root has
width at least two. Filling the second reserved slot of a $\mathcal D$-object
with an ordinary subtree gives a bijection
$\mathcal U\simeq\mathcal D\times\mathcal T$. Consequently,
\begin{equation}\label{DUseries}
 D=\sum_{w\geq2}g_w T^{w-2},\qquad U=TD,\qquad
 B=V+U,\qquad V=\frac{q}{1-q}.
\end{equation}
Here $V=g_1$ counts the width-one columns.

Within $\mathcal D$, distinguish the subclass $\mathcal K$ consisting of roots
whose column heights, from left to right, are
\begin{equation}\label{exceptionalcolumns}
 (h,2,\underbrace{1,\ldots,1}_{w-2}),\qquad h\geq2,\quad w\geq2,
\end{equation}
with all the $w-2$ ordinary subtrees equal to the empty-decorated tree. Each
such root has Durfee square of side two and second-largest row of length two,
so it is admissible. Its area is $h+w$. Thus
\begin{equation}\label{Kseries}
 K(q)=\sum_{h,w\geq2}q^{h+w}=\frac{q^4}{(1-q)^2}.
\end{equation}
Let $\mathcal D^\star$ consist of the other $\mathcal D$-objects: the root is
not of the form \eqref{exceptionalcolumns}, or at least one attached subtree
has positive area. Its generating function is $D^\star=D-K$.

\begin{Thm}\label{inverse-square}
The following identity holds:
\begin{equation}\label{positiveC}
 \boxed{\;1-T(q)^{-2}=2q+q^2+K(q)+2D^\star(q)+U(q)^2.\;}
\end{equation}
Thus $1-T^{-2}$ enumerates the disjoint union of two distinguishable
one-square objects, one fixed domino, $\mathcal K$, two distinguishable copies
of $\mathcal D^\star$, and ordered pairs of $\mathcal U$-objects, all weighted
by total area. The construction recovers Sokal's result that its coefficient
of $q^3$ is zero and all its other positive-degree coefficients are
strictly positive.
\end{Thm}

\begin{proof}
Put $a=1-V=(1-2q)/(1-q)$. Splitting the root equation by width gives
\begin{equation}\label{splitquadratic}
 aT=1+T^2D,\qquad T^{-1}=a-U,\qquad aU=D+U^2.
\end{equation}
It follows that
\begin{align}
 1-T^{-2}
 &=1-(a-U)^2=1-a^2+2D+U^2\notag\\
 &=(1-a^2+2K)+2D^\star+U^2.\label{Ccancel}
\end{align}
The remaining cancellation is given by
\begin{equation}
 1-a^2=\frac{2q-3q^2}{(1-q)^2}=2q+q^2-K(q),
\end{equation}
which proves \eqref{positiveC}. The two one-square objects and the
fixed domino are adjoined separately to the classes of reserved-slot
objects.

An admissible diagram of width at least two has Durfee side at
least two, and therefore area at least four. Thus $\mathcal D^\star$ and
$\mathcal U$ have no objects of area below four, and $\mathcal K$ starts in
area four. There is no object of area three. The two one-square objects
and the domino account for degrees one and two, while
$[q^m]K=m-3>0$ for $m\geq4$ proves strict positivity in every later degree.
\end{proof}

The first few terms are
\begin{equation}
 1-T^{-2}=2q+q^2+q^4+2q^5+7q^6+18q^7+50q^8+138q^9+\cdots.
\end{equation}
The construction gives the coefficientwise lower bound
$2q+q^2+q^4/(1-q)^2$ used in Sokal's proof
\cite[Eqn.~(5.5)]{Sokal2012}, together with an explicit class for the remainder.

\begin{Rem}
Sokal's discussion of real powers
\cite[Eqn.~(7.3) and Proposition~7.1]{Sokal2012} shows, in particular,
that for positive integers $r$ the series $1-T^{-r}$ has non-negative
coefficients only for $r=1,2$. Indeed,
\begin{equation}
 [q^2](1-T^{-r})=\frac{r(3-r)}2,
 \qquad [q^3](1-T^{-3})=-2.
\end{equation}
Thus no positive integer $r\geq3$ can give an everywhere non-negative
coefficient sequence for $1-T^{-r}$.
\end{Rem}

\subsection{Euler-product exponents}

The sequence decomposition also gives a combinatorial interpretation of the
exponents in Sokal's Euler product \cite[Conjecture 1.4]{Sokal2012}.

\begin{Cor}\label{Euler}
There are unique integers $c_m$ such that
\begin{equation}\label{Eulerproduct}
 T(q)=\prod_{m\geq1}(1-q^m)^{-c_m}.
\end{equation}
They count Lyndon words in the alphabet $\mathcal B$, weighted by total area.
In particular, $c_m\geq1$ for every $m\geq1$ and
$c_{m+1}\geq c_m$ for every $m\geq1$.
\end{Cor}

\begin{proof}
Choose a total order on the $\mathcal B$-objects with the unique area-one
object $b$ first. A Lyndon word is a non-empty word strictly smaller in
lexicographic order than each of its nontrivial cyclic rotations; in
particular it is not a proper power. The Chen--Fox--Lyndon factorization theorem
\cite[p.~582]{MelanconReutenauer1989} expresses every word uniquely as a
non-increasing sequence of Lyndon words. Applied to
$\mathcal T\simeq\operatorname{SEQ}(\mathcal B)$, this gives
\eqref{Eulerproduct}. All constructions are locally finite because the letters
have positive integer weights and there are finitely many of each weight.
Uniqueness of the integers $c_m$ follows also by comparing successive
coefficients in the product.

There is a one-letter Lyndon word of every weight, represented by a column,
so $c_m\geq1$. If $m\geq2$ and $L$ is a Lyndon word of weight $m$, then
$b<L$. The concatenation property for Lyndon words
\cite[p.~578]{MelanconReutenauer1989} implies that $bL$ is Lyndon. Prefixing $b$ is injective and increases weight by one.
Finally $c_1=c_2=1$, since the only other word of weight two is $bb$, which is
not Lyndon.
\end{proof}

This settles the positivity and weak-monotonicity assertions of
Sokal's Conjecture~1.4. Its remaining assertions concern strict convexity and the
nonnegativity of the third and fourth differences.
The same two conclusions follow from his Theorem~1.2 by taking an
alphabet with $[q^m](1-T^{-1})$ letters of weight $m$ and applying
Lyndon factorization \cite{Reutenauer1993} and the prefix injection above.
Here the letters are the explicitly described objects of $\mathcal B$.
The exponents begin
\begin{equation}
 c_1,c_2,\ldots=1,1,2,4,10,23,61,157,426,1163,\ldots.
\end{equation}

\section{Mixed levels, forests, and braids}
\label{secmixed}

An immediate consequence of Theorems \ref{thmA} and \ref{thmB} is the following corollary.

\begin{Cor}
$S_q$-enriched rooted trees with fixed total area $A$ and $F_q$-enriched rooted trees with fixed total area $A$ are equinumerous.
\end{Cor}

However, this is just a special case of a much more general result. Given that the tree structures in Lemmas \ref{lemstackgen} and \ref{lemferrers} are related to the iteration of the two functional equations (\ref{root1}) and (\ref{root2}) for $\xi_0(q)$ from Lemma~\ref{sokal}, respectively, one can generalise Theorems \ref{thmA} and \ref{thmB}
immediately to an infinite family of equinumerous sets of trees.

To see this, note that $\xi_0(q)$ satisfies not only

\begin{align}
\xi_0(q)&=F(\xi_0(q),q)\;\text{, and}\\
\xi_0(q)&=\tilde F(\xi_0(q),q)\;,
\end{align}
but also
\begin{align}
\xi_0(q)&=F(F(\xi_0(q),q),q)\;,\\
\xi_0(q)&=F(\tilde F(\xi_0(q),q),q)\;,\\
\xi_0(q)&=\tilde F(F(\xi_0(q),q),q)\;,\\
\xi_0(q)&=\tilde F(\tilde F(\xi_0(q),q),q)\;,
\end{align}
and so forth, and iteration of each of these leads to a different combinatorial model. More generally, I obtain Theorem \ref{main}.

\begin{Thm}
\label{main}
Let $F^{(0)}_q(a)=1+G(1,1,a,q)$ be the generating function of stack polyominoes enumerated with respect to rise ($a$) and total area ($q$), and let $F^{(1)}_q(a)=1+\tilde G(a,1,q)$ be the generating function of Ferrers diagrams with $n$-th largest row having length $n$ for some integer $n$, enumerated with respect to width ($a$) and total area ($q$).

Let $\sigma=(\sigma_0,\ldots,\sigma_N)\in\{0,1\}^{N+1}$ for $N\geq0$. Then
\beq
\label{enum}
\xi_0^{\sigma}(q)=F^{(\sigma_0)}_q\circ F^{(\sigma_1)}_q\circ\ldots\circ F^{(\sigma_N)}_q(0)
\eeq
enumerates rooted trees of height at most $N$, enriched by $S_q$ at level $i$ if $\sigma_i=0$ 
and enriched by $F_q$ at level $i$ if $\sigma_i=1$, weighted with respect to area (level 0 is the root).

Moreover, given $\sigma\in\{0,1\}^{\mathbb N}$, $\xi_0(q)$ enumerates rooted trees
enriched by $S_q$ at level $i$ if $\sigma_i=0$ and enriched by $F_q$ at level $i$ if $\sigma_i=1$, 
weighted with respect to area. In particular, sets of trees enriched with respect to any $\sigma\in\{0,1\}^{\mathbb N}$ are equinumerous for fixed total area $A$.
\end{Thm}

\begin{proof}
Let $R^{(0)}_q=S_q$ be the species of stack polyominoes augmented by the `empty polyomino', weighted by area ($q$), with size given by the rise, and let $R^{(1)}_q=F_q$ be the species of Ferrers diagrams with $n$-th largest row having length $n$ for some integer $n$, augmented by the `empty polyomino', weighted by area ($q$), with size given by the width. Let $\sigma\in\{0,1\}^{N+1}$. Then the combinatorial
expression
\beq
X\cdot R^{(\sigma_0)}_q(X\cdot R^{(\sigma_1)}_q(\ldots X\cdot R^{(\sigma_N)}_q(0)\ldots))
\eeq
encodes trees of height at most $N$, enriched by $R^{(\sigma_i)}_q$ at level $i$. Hence
\beq
A_{\sigma}(t,q)=tF^{(\sigma_0)}_q(tF^{(\sigma_1)}_q(\ldots tF^{(\sigma_N)}_q(0)\ldots))
\eeq
enumerates these trees with respect to number of vertices ($t$) and total area ($q$).
Letting $t=1$, Eqn.~(\ref{enum}) follows.

For an infinite level sequence, the finite-height classes are nested and
exhaust the class of all finite trees with those level enrichments.
Coefficientwise finiteness follows from $v\leq A+1$ as above. To identify the
limit, take a finite prefix $(\sigma_0,\ldots,\sigma_N)$ with $n_0$ zero
entries and $n_1$ one entries. Both maps fix $\xi_0$, so repeated use of
\eqref{contractions} gives
\begin{equation}\label{mixed-error}
 F_q^{(\sigma_0)}\circ\cdots\circ F_q^{(\sigma_N)}(0)
       =\xi_0(q)+O(q^{3n_0+n_1}).
\end{equation}
Since $3n_0+n_1\geq N+1$, the limit is $\xi_0$ for every infinite choice of
levels. The result follows.
\end{proof}

\begin{figure}[p]
\begin{center}\includegraphics[width=0.6\textwidth]{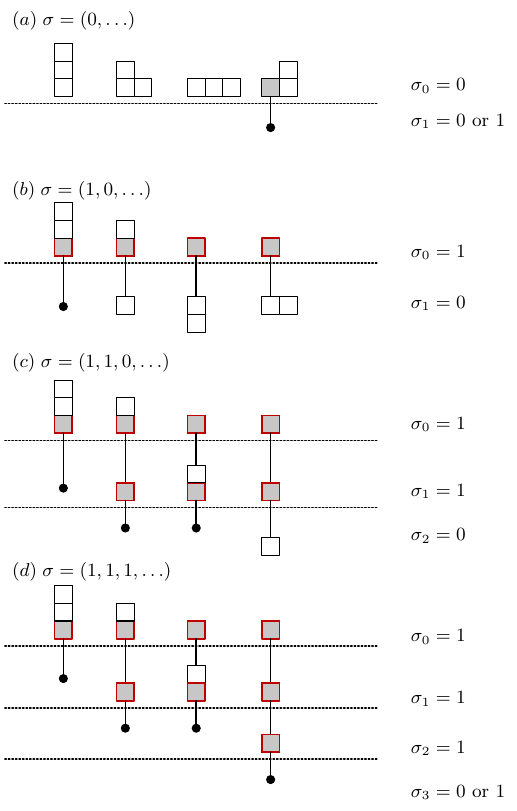}
\begin{minipage}{0.9\textwidth}\caption{Enriched rooted trees with total area 3 for
different choices of $\sigma$: (a) trees enriched with stack polyominoes at level 0, (b) trees enriched with Ferrers diagrams at level 0 and stack polyominoes at level 1,
(c) trees enriched with Ferrers diagrams at levels 0 and 1, and stack polyominoes at level 2, (d) trees enriched with Ferrers diagrams at levels 0, 1, and 2.
\label{mixed}
}\end{minipage}\end{center}
\end{figure}

Figure \ref{mixed} shows enriched rooted trees with total area 3 for different
choices of $\sigma\in\{0,1\}^{\mathbb{N}}$. For each choice of $\sigma$, there are four enriched trees
with total area 3. Note that the number of vertices in the trees differ, corresponding to different expressions for the coefficient polynomials of $q^3$ in $A_{\sigma}(t,q)$:

\begin{itemize}
\item[(a)] $\sigma=(0,\ldots)$ implies
\beq 
A_{\sigma}(t,q)=t+tq+2tq^2+(t^2+3t)q^3+\ldots\;,
\eeq
\item[(b)] $\sigma=(1,0,\ldots)$ implies
\beq 
A_{\sigma}(t,q)=t+t^2q+2t^2q^2+4t^2q^3+\ldots\;,
\eeq
\item[(c)] $\sigma=(1,1,0\ldots)$ implies
\beq 
A_{\sigma}(t,q)=t+t^2q+(t^3+t^2)q^2+(3t^3+t^2)q^3+\ldots\;,
\eeq
and
\item[(d)] $\sigma=(1,1,1,\ldots)$ implies
\beq 
A_{\sigma}(t,q)=t+t^2q+(t^3+t^2)q^2+(t^4+2t^3+t^2)q^3+\ldots\;.
\eeq
\end{itemize}

\subsection{Finite iterations and boundary conventions}
Theorem~\ref{main} gives a combinatorial interpretation of the iterations in
Remarks 4 and 5 after Proposition 3.1 of \cite{Sokal2012}: they are
$(F_q^{(0)})^n(0)$ and $(F_q^{(1)})^n(1)=(F_q^{(1)})^{n+1}(0)$,
respectively. Here the last equality uses $F_q^{(1)}(0)=1$.
The iteration in Proposition 3.1 is
\begin{equation}\label{iteration-one}
 \xi_0^{(n)}(q)=(F_q^{(0)})^n(1).
\end{equation}
For $n\geq1$, this counts stack-enriched trees with levels $0,\ldots,n-1$,
with the ordinary slot condition at levels below $n-1$ and with successor
slots at level $n-1$ suppressed. A terminal decoration at that last level
may therefore be any stack, including the empty stack, regardless of its
rise. Earlier leaves must still have rise zero. This corresponds to the
terminal substitution of $1$. The same convention applies to mixed finite
iterations.

\subsection{Ordered forests and the braid interpretation}\label{secforests}
\begin{Cor}\label{forests}
Fix a positive integer $r$ and an infinite sequence of level enrichments for
each of $r$ ordered components. Ordered $r$-tuples of the corresponding
enriched trees have area generating function $T(q)^r$.
\end{Cor}
\begin{proof}
Each component has generating function $T$ by Theorem~\ref{main}, and areas
add under the ordered Cartesian product.
\end{proof}

Flores and Gonz{\'a}lez-Meneses
\cite[Propositions 6.5--6.7]{FloresGonzalezMeneses2022}
give the same series a braid-theoretic interpretation. The positive braid
monoid $A_\infty$ has generators $a_1,a_2,\ldots$ and relations
$a_i a_j=a_j a_i$ for $|i-j|\geq2$ and
$a_i a_{i+1}a_i=a_{i+1}a_i a_{i+1}$. Order the generators by
$a_1<a_2<\cdots$. The coefficient $[q^k]T(q)^r$ counts the length-$k$
elements whose maximal lexicographic representative starts in
$\{a_1,\ldots,a_r\}$, together with the empty braid when $k=0$.
Their decomposition into $r$ components explains the power $T^r$.
Corollary~\ref{forests} therefore gives an equinumerous family of ordered
forests. It would be interesting to find a direct bijection between these
forests and the corresponding braid classes.

Flores and Gonz{\'a}lez-Meneses also identify the limiting exponential growth
rates of the finite-rank monoids with the partial-theta constant
\cite{FloresGonzalezMeneses2022}. This constant appears in the coefficient
asymptotics below.

\section{Singularity analysis and coefficient asymptotics}
\label{secasymptotics}

Sokal \cite[Section~7, Eqn.~(7.5)]{Sokal2012} identifies the first
singularity of the leading-root series with the collision of its first
two roots and gives the approximate collision point.
Kostov \cite[Theorem~1]{KostovClosest2016} proves that
$0.3092493386\ldots$ is the only spectral value in $|q|\leq0.31$.
I shall give a complete derivation from the positive stack equation,
including the local square-root expansion, its amplitude and the
analytic continuation needed for coefficient asymptotics.
In this section $F(w,q)$ denotes the stack map
\begin{equation}\label{analyticF}
 F(w,q)=1+\sum_{k\geq1}
 \frac{q^k}{(q;q)_k(wq;q)_{k-1}}.
\end{equation}
This function is analytic for $|q|<1$ and $|qw|<1$: on compact subsets of
this domain the products in the denominators are uniformly bounded away
from zero and the summands have a geometric majorant.

\begin{Thm}\label{asymptotics}
Let $\rho$ be the radius of convergence of $T(q)=\xi_0(q)$ and put
$\tau=T(\rho)$. Then
\begin{equation}\label{characteristic}
 0<\rho<\frac13,\qquad 1<\tau<\frac1\rho,\qquad
 \tau=F(\tau,\rho),\qquad F_w(\tau,\rho)=1.
\end{equation}
The point $\rho$ is the only singularity of $T$ on its circle of convergence.
In a dented neighbourhood of $\rho$,
\begin{equation}\label{sqrtlocal}
 T(q)=\tau-\tau\sqrt{1-q/\rho}+O(1-q/\rho),
\end{equation}
with a convergent expansion in powers of $\sqrt{1-q/\rho}$. Consequently,
\begin{equation}\label{coefficient-asymptotics}
 [q^n]T(q)=\frac{\tau}{2\sqrt{\pi}}\,
       \rho^{-n}n^{-3/2}\bigl(1+O(n^{-1})\bigr).
\end{equation}
\end{Thm}

\begin{proof}
The analytic implicit-function theorem at $(w,q)=(1,0)$ gives a positive
radius of convergence. The combinatorial interpretation shows that $T$ has
non-negative coefficients. For $0<q<\min\{\rho,1\}$, expand the reciprocal
denominator factors in the formal identity $T=F(T,q)$ as geometric
series. All terms are non-negative, so Tonelli's theorem permits
regrouping after evaluation; their sum is the finite number $T(q)$.
In particular, the geometric subseries in the $k=2$ term is finite,
which forces $qT(q)<1$ and gives
\begin{equation}\label{lower-tau}
 T(q)>\frac1{1-q}+
 \frac{q^2}{(1-q)(1-q^2)(1-qT(q))}.
\end{equation}
The inequality is strict because the terms with $k\geq3$ are positive.
If $\rho>1/3$, evaluate at $q=1/3$ and write $t=T(1/3)$. Then
\begin{equation}
 (t-\tfrac32)(1-t/3)>\tfrac3{16},
\end{equation}
which is impossible: the quadratic on the left has maximum $3/16$ at
$t=9/4$. Hence $\rho\leq1/3$.

As $q\uparrow\rho$, monotonicity and $T(q)<1/q$ give a finite limit
$\tau\leq1/\rho$. By monotone convergence of the non-negative series,
$T(\rho)=\tau$. The $k=2$ term also gives
\begin{equation}
 1-qT(q)\geq
 \frac{q^2}{(1-q)(1-q^2)T(q)}.
\end{equation}
Passing to the limit proves $\rho\tau<1$. Thus $F$ is analytic at
$(\tau,\rho)$, and continuity gives $\tau=F(\tau,\rho)$. If
$\rho=1/3$, the strictly positive $k=3$ term allows the same strict
inequality \eqref{lower-tau} at the limit, again contradicting the quadratic
bound. Therefore $\rho<1/3$.

Pringsheim's theorem implies that $\rho$ is a singularity of $T$. If
$F_w(\tau,\rho)\ne1$, the implicit-function theorem would give a local
analytic solution through $(\tau,\rho)$. Its uniqueness identifies it
with $T$ for real $q<\rho$ sufficiently close to $\rho$, and hence
continues $T$ analytically through $\rho$, a contradiction. This proves the
characteristic equations. The non-negative coefficient expansion of $F$
gives $F_q(\tau,\rho)>0$ and $F_{ww}(\tau,\rho)>0$.
To apply the smooth implicit-function schema
\cite[Definition VII.4 and Theorem VII.3]{Flajolet2009}, put
$Y=T-1$ and $H(q,u)=F(1+u,q)-1$. Then $Y=H(q,Y)$,
$Y(0)=H(0,0)=H_u(0,0)=0$, and $H$ has non-negative coefficients and
nonlinear dependence on $u$. Since $\rho\tau<1$, there are
$R>\rho$ and $S>\tau-1$ with $R(1+S)<1$, so $H$ is analytic in a
bidisc containing the characteristic point $(\rho,\tau-1)$.
The schema yields a convergent local Puiseux expansion
\begin{equation}\label{generic-amplitude}
 T(q)=\tau-c\sqrt{1-q/\rho}+O(1-q/\rho),\qquad
 c=\sqrt{\frac{2\rho F_q(\tau,\rho)}{F_{ww}(\tau,\rho)}}.
\end{equation}
The negative sign is selected by $T(q)<\tau$ on $0<q<\rho$.

To evaluate the constant $c$, I use the identity from
Lemma~\ref{sokal}, which holds in the present analytic domain:
\begin{equation}\label{theta-factor}
 \Theta_0(-w,q)=(q;q)_\infty(wq;q)_\infty\bigl(F(w,q)-w\bigr).
\end{equation}
The prefactor is nonzero at $(\tau,\rho)$, so this point corresponds to a
double zero in the $x$ variable:
$\Theta_0(-\tau,\rho)=\Theta_{0,x}(-\tau,\rho)=0$ and
$\Theta_{0,xx}(-\tau,\rho)\ne0$. Termwise differentiation gives
\begin{equation}\label{heat}
 2q\Theta_{0,q}(x,q)=x^2\Theta_{0,xx}(x,q).
\end{equation}
At the characteristic point, differentiation of \eqref{theta-factor} gives
$\Theta_{0,q}/\Theta_{0,xx}=F_q/F_{ww}$. Equation~\eqref{heat} therefore
shows that $2\rho F_q/F_{ww}=\tau^2$, proving $c=\tau$.

It remains to exclude other dominant singularities. Since
$T(\rho)=\tau<\infty$, its series converges absolutely on $|q|=\rho$. The coefficients of $q^0$ and $q^1$ are both one,
so strictness in the triangle inequality gives
$|T(z)|<\tau$ for $|z|=\rho$, $z\ne\rho$. The identity $T=F(T,z)$
continues to these boundary points by continuity. Nonnegativity of the
coefficients of $F_w$ and strict increase in its first positive argument give
\begin{equation}
 |F_w(T(z),z)|\leq F_w(|T(z)|,\rho)<F_w(\tau,\rho)=1.
\end{equation}
Also $|zT(z)|<1$, so the implicit-function theorem continues $T$ analytically
through every such point. Compactness away from $\rho$, together with the
local Puiseux expansion, gives a $\Delta$-domain at $\rho$.
Singularity transfer \cite[Theorem VI.4]{Flajolet2009}, applied to the full
Puiseux expansion, proves \eqref{coefficient-asymptotics}.
\end{proof}

The positive characteristic pair in \eqref{characteristic} is unique within
$0<q<1/3$, $w>0$, $qw<1$. Indeed, $F(w,q)-w$ is strictly convex in $w$,
and a zero with vanishing $w$ derivative is its global minimum. Two such
pairs with $q_1<q_2$ would contradict the strict increase of $F(w,q)$ in $q$,
by evaluating the earlier function at the later minimum.

Solving the double-zero equations numerically gives, consistently with
Sokal's estimates \cite[Section~7]{Sokal2012},
\begin{equation}\label{numerical-characteristic}
 \rho\simeq 0.309249338600,\qquad
 \tau\simeq 2.320376944290,\qquad
 \rho^{-1}\simeq 3.233636665245.
\end{equation}
The last constant is the partial-theta growth constant occurring in
\cite{KostovShapiro2013,FloresGonzalezMeneses2022}.

\begin{Cor}\label{forest-asymptotic}
For every fixed positive integer $r$,
\begin{equation}
 [q^n]T(q)^r\sim
 \frac{r\tau^r}{2\sqrt\pi}\rho^{-n}n^{-3/2}.
\end{equation}
The inverse-power models satisfy
\begin{equation}
 [q^n](1-T^{-1})\sim\frac{1}{2\tau\sqrt\pi}\rho^{-n}n^{-3/2},
 \qquad
 [q^n](1-T^{-2})\sim\frac{1}{\tau^2\sqrt\pi}\rho^{-n}n^{-3/2}.
\end{equation}
\end{Cor}

\begin{proof}
Apply the same transfer argument to \eqref{sqrtlocal}. For inverse powers,
$T$ is nonzero throughout its closed convergence disc: the identity
$\Theta_0(-T(q),q)=0$ holds there by continuity, whereas
$\Theta_0(0,q)=1$. It remains nonzero on a sufficiently small continuation
domain around that disc, so the same transfer argument applies to the
inverse powers.
\end{proof}

\begin{Rem}[The limiting braid growth rate]\label{braid-growth-repair}
Write $L_k=[q^k]T(q)$ and $Q=\rho^{-1}$. Theorem~\ref{asymptotics}
gives $L_k/Q^k\to0$, and
\[
 \sum_{k\geq1}L_k/Q^k=T(\rho)-1=\tau-1<\infty.
\]
Thus the positive limit asserted in
\cite[Eqn.~(4), p.~34]{FloresGonzalezMeneses2022}, and the divergence
used at the end of Proposition~6.15 there, do not hold. Their conclusion
that the growth rates $\lambda_n$ of the finite-rank braid monoids tend
to $Q$ follows from their uniform Perron-vector bound, as follows.
The notation $\lambda_n$, $\lambda=\lim_n\lambda_n$ and $Q$ used here
corresponds to $\rho_n$, $\rho$ and $q_\infty$ in that paper.

Let $R_n$ be the incidence matrix of the recurrent component
$\Gamma_n^*$ of the automaton of maximal lexicographic representatives,
with rows indexing initial states of transitions. This is the primitive
matrix of \cite[Section~5 and Appendix]{FloresGonzalezMeneses2020},
whose Perron eigenvalue is $\lambda_n$. The component is closed under
transitions and contains $t_{1,1}=F_n(a_1)$, the state reached after
reading $a_1$. Every state is accepting. Consequently, if $e_1$ selects
$t_{1,1}$ and $e$ is the all-ones column vector, paths of length $k$
from that state give
\[
 e_1R_n^ke=m_{k+1,1}(A_n),
\]
where $m_{k+1,1}(A_n)$ counts the length-$(k+1)$ representatives
starting with $a_1$.

Let $v_n$ be the left Perron vector, normalised by $v_ne=1$.
Corollary~5.5 of \cite{FloresGonzalezMeneses2020} gives the uniform
bound $(v_n)_1>1/32$, quoted in
\cite[Theorem~6.10]{FloresGonzalezMeneses2022}. Its proof uses the
automaton structure and the a priori bound $\lambda_n<4$, independently
of the limiting-growth conclusion. The published Appendix of
\cite{FloresGonzalezMeneses2020} records the stronger bound $3/64$;
$1/32$ suffices here. For every $k\geq1$,
\[
 \lambda_n^k=v_nR_n^ke
 \geq (v_n)_1e_1R_n^ke
 >\frac{m_{k+1,1}(A_n)}{32}.
\]
By \cite[Proposition~6.1]{FloresGonzalezMeneses2022}, the last count
equals $L_{k+1}$ for $n\geq k+1$. The sequence $\lambda_n$ is
non-decreasing by inclusion of the monoids. For fixed $n$, primitivity
gives $\lim_{j\to\infty}m_{j,1}(A_n)^{1/j}=\lambda_n$;
since $m_{j,1}(A_n)\leq L_j$, it follows that $\lambda_n\leq Q$.
Thus $\lambda=\lim_n\lambda_n\leq Q$ exists. Taking $n\to\infty$
for fixed $k$ gives $\lambda^k\geq L_{k+1}/32$. Taking $k$-th roots
and then the limit superior in $k$ yields $\lambda\geq Q$.
Therefore $\lambda=Q$. This last passage uses only
$\limsup_j L_j^{1/j}=Q$, which follows from the radius of convergence.
\end{Rem}

\Needspace{10\baselineskip}
\section{Further questions}

It would be interesting to find direct bijections between the tree models
for different infinite level sequences, and between these trees or forests
and the corresponding braid classes.

The Lyndon-word interpretation in Corollary~\ref{Euler} suggests studying
Sokal's conjectures about finite differences of the Euler-product exponents
and the log-convexity of the coefficients of $T$
\cite[Conjectures 1.4 and 1.6]{Sokal2012}. It would be interesting to find
combinatorial proofs of these properties.

The variable $t$ in $A_\sigma(t,q)$ counts vertices. It is tempting to ask
whether there is a generalised theta function with $-A_\sigma(t,q)$ as its
leading root, or whether the vertex variable has an interpretation in terms
of braid statistics.

\section*{Acknowledgements}
I am grateful to Andrew Rechnitzer and Alan Sokal for useful discussions and comments on an earlier version of this manuscript. I also thank the referees for helpful suggestions.

\end{document}